\documentclass{amsart}

\usepackage{amsmath,amsthm,amssymb,latexsym,amstext,amsfonts,mathabx}
\usepackage[english]{babel}
\usepackage{geometry}
\usepackage{graphics}

\usepackage{tikz}
\usetikzlibrary{%
  matrix,%
  calc,%
  arrows,%
  shapes,%
  positioning%
}
\usetikzlibrary{positioning}

\usepackage{hyperref}
\usepackage{caption}

\usepackage{float}
\floatstyle{boxed} 
\restylefloat{figure}

\usepackage{ifthen,xkeyval,xcolor,calc,graphicx}
\usepackage{todonotes}

\newtheorem{theorem}{Theorem}[section]

\newtheorem{corollary}[theorem]{Corollary}

\newtheorem{example}[theorem]{Example}

\newtheorem{lemma}[theorem]{Lemma}

\newtheorem{proposition}[theorem]{Proposition}

\def \RR{{\mathcal R}}

\def \LL{{\mathcal L}}
\def \DD{{\mathcal D}}

\def \VV{{\mathcal V}}

\newcommand{\Drel}{\mathbin{\mathcal D}}

\newcommand{\set}[1]{\{\,#1\,\}}

\begin{document}

\title{\textbf{Distributivity in Skew Lattices}}

\author{Michael Kinyon}
\address{University of Denver, USA}

\author{Jonathan Leech}
\address{Westmont College. USA}

\author{Jo\~ao Pita Costa} 
\address{In\v stitut Jozef \v Stefan, Slovenia}

\date{\today}

\maketitle


\begin{abstract}
Distributive skew lattices satisfying $x\wedge (y\vee z)\wedge x = (x\wedge y\wedge x) \vee (x\wedge z\wedge x)$ and its dual are studied, along with the larger class of linearly distributive skew lattices, whose totally preordered subalgebras are distributive. 
Linear distributivity is characterized in terms of the behavior of the natural partial order between comparable $\DD$-classes. 
This leads to a second characterization in terms of strictly categorical skew lattices. 
Criteria are given for both types of skew lattices to be distributive.
\\


\end{abstract}


\section{Introduction}
\label{Introduction}
 
Recall that a lattice $(L; \wedge, \vee)$ is \emph{distributive} if the identity $x\wedge (y \vee z) = (x\wedge y) \vee (x\wedge z)$ holds on $L$. 
One of the first results in lattice theory is the equivalence of this identity to its dual, $x\vee(y\wedge z) = (x\vee y) \wedge (x\vee z)$. 
Distributive lattices are also characterized as being \emph{cancellative} in that $x\wedge y = x\wedge z$ and $x\vee y = x\vee z$ jointly imply $y = z$. 
A third characterization is that neither of the 5-element lattices below can be embedded in the given lattice.

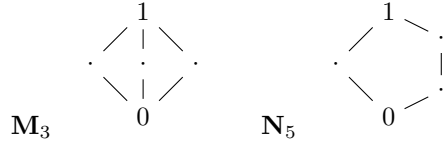
\begin{center}
$\begin{array}{ccccc}

\mathbf{M}_{3}

&

\begin{tikzpicture}[scale=.7]

  \node (1) at (0,1){$1$} ;
  \node (a) at (-1,0){$\cdot$} ;
  \node (b) at (0,0){$\cdot$};
  \node (c) at (1,0){$\cdot$}  ;
  \node (0) at (0,-1){$0$} ;
  
    \draw (1) -- (c) -- (0) -- (a) -- (1) -- (b) -- (0);

\end{tikzpicture}

&

&

\mathbf{N}_{5}

&

\begin{tikzpicture}[scale=.7]

  \node (1) at (0,1){$1$} ;
  \node (a) at (-1,0){$\cdot$} ;
  \node (b) at (1,0.5) {$\cdot$} ;
  \node (c) at (1,-0.5) {$\cdot$} ; 
  \node (0) at (0,-1) {$0$} ;
  
      \draw (1) -- (a) -- (0) -- (c) -- (b) -- (1) ;

\end{tikzpicture}\\

\end{array}$
\end{center}

Distributivity also arises when studying \emph{skew lattices}, that is, algebras with associative, idempotent binary operations $\vee $ and $\wedge $ that satisfy the absorption identities:

\begin{equation}\label{absidentities}\tag{1.1}
x\wedge (x\vee y) = x = (y\vee x)\wedge x \text{  and  } x\vee (x\wedge y) = x = (y\wedge x)\vee x.  
\end{equation} %

\noindent Given that $\wedge$ and $\vee$ are associative and idempotent, (\ref{absidentities}) is equivalent to the dualities:

\begin{equation}\label{absequivalences}\tag{1.2}
x\wedge y=x \text{  iff  } x\vee y=y \text{  and  } x\wedge y=y \text{  iff  } x\vee y=x. 
\end{equation} %

\noindent For skew lattices, the distributive identities of greatest interest have been the dual pair:

\begin{equation}\label{GMD}\tag{1.3}
x\wedge (y\vee z)\wedge x = (x\wedge y\wedge x)\vee (x\wedge z\wedge x);  
\end{equation} %

\begin{equation}\label{GJD}\tag{1.4}
x\vee (y\wedge z)\vee x = (x\vee y\vee x)\wedge (x\vee z\vee x).  
\end{equation} %

\noindent Indeed, a skew lattice is \emph{distributive} if it satisfies both. Unlike the case of lattices, (\ref{GMD}) and (\ref{GJD}) are not equivalent. 
Spinks, however, obtained a computer proof in \cite{Sp00} (humanized later by Cvetko-Vah in \cite{Ka06}) of their equivalence for skew lattices that are \emph{symmetric} in that:

\begin{equation}\label{sym}\tag{1.5} 
x\wedge y = y\wedge x \text{  iff  } x\vee y = y\vee x. 
\end{equation}%

\noindent (See \cite{Ka06}, \cite{Sp00R}, \cite{Sp00}.) 
Also unlike lattices, distributive skew lattices need not be \emph{cancellative} in that they need not satisfy:

\begin{equation}\label{canc}\tag{1.6}
x\vee y=x\vee z \text{  and  } x\wedge y=x\wedge z \text{  imply  } y=z \text{,  and  } x\vee z=y\vee z \text{  and  } x\wedge z=y\wedge z \text{  imply  } x=y.
\end{equation} %

Conversely, cancellative skew lattices need not be distributive, but they are always symmetric, unlike distributive skew lattices. $\mathbf M_{3}$ and $\mathbf N_{5}$ are forbidden subalgebras of both types of algebras. 
Their absence is equivalent to the weaker condition of being \emph{quasi-distributive} in that the skew lattice has a distributive maximal lattice image. (See \cite{Ka11c} Theorem 3.2.)
Of course, many skew lattices are both distributive and cancellative. This is true for skew Boolean algebras (\cite{Ba11}, \cite{BL}, \cite{BS}, \cite{Ku12}, \cite{Le90}, \cite{Le08}, \cite{Sp00}, \cite{Sp06}) and skew lattices of idempotents in rings (\cite{Ka05c}, \cite{Ka05}, \cite{Ka08}, \cite{Ka12}, \cite{Ka11}, \cite{Le89}, \cite{Le05}). 
Identities \eqref{GMD} and \eqref{GJD} also arise in studying broader types of noncommutative lattices. (See \cite{La02} Section 6.) 

Identities \eqref{GMD} and \eqref{GJD} insure that the maps $x \mapsto a\wedge x\wedge a$ and $x \mapsto a\vee x\vee a$ are homomorphic retractions of $\mathbf S$ onto the respective subalgebras $\set{x \in S\mid a\wedge x=x=x\wedge a}$ and $\set{x \in S\mid a\vee x=a=x\vee a}$ for each element $a$ in the skew lattice $\mathbf S$. In this paper we study further effects of being distributive, as well as connections between distributive skew lattices and other varieties of algebras. A main concept in our study is \emph{linear distributivity} which assumes that all subalgebras that are totally preordered under the natural preorder $\succeq$ as defined in \eqref{pre} below, are distributive. 
This is unlike the case for lattices where totally ordered sublattices are automatically distributive. Like quasi-distributivity, linear distributivity, is necessary but not sufficient for a skew lattice to be distributive.

We begin by reviewing some of the required background for this paper in Section \ref{Background}. (For more thorough remarks, see \cite{Le96} or introductory remarks in \cite{CAT}.) 
In Section \ref{Linear Distributivity} linear distributivity is introduced with characterizing identities given in Theorem \ref{lindist}. In the next section it is studied in terms of the natural partial order $\geq$ defined in \eqref{poset} below, with attention given to the behavior of $\geq$ on a skew chain of comparable $\DD$-classes, $A > B > C$. 
Distributive skew chains are characterized by the behavior of their \emph{midpoint sets} given  by $\mu(a, c) = \set{b\in B\mid a>b>c}$ for any pair $a>c$ with $a\in A$ and $c\in C$. While these sets often contain many midpoints, (\ref{GMD}) and (\ref{GJD}) minimize their size. The details are given in Section \ref{Midpoints and Distributive Skew Chains}, whose main result, Theorem \ref{disteq}, characterizes distributive skew chains (and by extension, linearly distributive skew lattices) not only in terms of midpoints but also in terms of strictly categorical skew lattices (first studied in \cite{CAT}, Section \ref{Midpoints and Distributive Skew Chains}).
The latter generalize both normal skew lattices (where $(S, \wedge)$ is a normal band) first studied in 
this journal \cite{Le92} and their $\wedge-\vee$ duals.

Is linear distributivity in concert with quasi-distributivity enough to guarantee that a skew lattice is distributive? In general, the answer is no. It is, however, for strictly categorical skew lattices, which form a significant subclass of linearly distributive skew lattices. (See Theorem \ref{strcat} and the relevant discussion in Section \ref{Midpoints and Distributive Skew Chains}.) 
If we assume that the skew lattice is symmetric, the answer is yes (Theorem \ref{distrib}). A characterization of those linearly distributive and quasi-distributive skew lattices that are distributive is given in Theorem \ref{distlindist}.
%


\section{Background}
\label{Background}

Returning first to symmetric skew lattices, they form a variety of skew lattices that is characterized by the following identities:

\begin{equation}\label{syma}\tag{2.1}
x\vee y\vee (x\wedge y) = (y\wedge x)\vee y\vee x 
\end{equation}  
\begin{equation}\label{symb}\tag{2.2}
x\wedge y\wedge (x\vee y) = (y\vee x)\wedge y\wedge x 
\end{equation}  

\noindent given first by Spinks \cite{Le05}.
The identity \eqref{syma} characterizes \emph{upper symmetry} ($x \wedge y = y \wedge x$ implies $x \vee y = y \vee x$ for all $x, y\in S$) while the identity \eqref{symb} characterizes lower symmetry ($x \vee y = y \vee x$ implies $x \wedge y = y \wedge x$ for all $x, y\in S$).

The \emph{GreenÕs relations} are defined on a skew lattice by

\begin{equation}\label{R}\tag{2.3R} a\RR b \Leftrightarrow (a\wedge b = b \text{  and  } b\wedge a = a) \Leftrightarrow (a\vee b = a \text{  and  } b\vee a = b);\end{equation}
\begin{equation}\label{L}\tag{2.3L} a\LL b \Leftrightarrow (a\wedge b = a \text{  and  } b\wedge a = b) \Leftrightarrow (a\vee b = b \text{  and  } b\vee a = a);\end{equation}
\begin{equation}\label{D}\tag{2.3D} a\DD b \Leftrightarrow (a\wedge b\wedge a = a \text{  and  } b\wedge a\wedge b = b) \Leftrightarrow (a\vee b\vee a = a \text{  and  } b\vee a\vee b = b).\end{equation}

All three relations are canonical congruences, with $\LL \vee \RR = \LL \circ \RR = \RR \circ \LL = \DD$ and $\LL \cap \RR = \Delta=\set{(x,x)\mid x\in S }$, the identity equivalence.
Their congruence classes are called $\DD$-classes, $\LL$-classes or $\RR$-classes and are often denoted by $\DD_{x}$, $\LL_{x}$ or $\LL_{x}$ where $x$ is some class member.

A skew lattice $\mathbf S$ is \emph{rectangular} if $x\wedge y\wedge x = x$, or dually $y\vee x\vee y = y$, holds on $\mathbf S$. 
Such a skew lattice is anti-commutative in that $x\wedge y = y\wedge x$ or $x\vee y = y\vee x$ imply $x = y$. 
The \emph{First Decomposition Theorem} (see \cite{Le89} Theorem 1.7) states that \emph{in any skew lattice $\mathbf S$ each $\DD$-congruence class is a maximal rectangular subalgebra of $\mathbf S$ and $\mathbf S/\DD$ is the maximal lattice image of $\mathbf S$}. In particular, a rectangular skew lattice consists of a single $\DD$-class. 
A skew lattice is \emph{right-handed} [respectively \emph{left-handed}] if it satisfies the identities 

\begin{equation}\label{RH}\tag{2.4R} 
x\wedge y\wedge x = y\wedge x\text{  and  } x\vee y\vee x = x\vee y  
\end{equation}
\begin{equation}\label{LH}\tag{2.4L}
[x\wedge y\wedge x = x\wedge y \text{  and  } x\vee y\vee x = y\vee x].  
\end{equation}

\noindent Equivalently, $x\wedge y = y$ and $x\vee y = x$ [$x\wedge y = x$ and $x\vee y = y$] hold in each $\DD$-class, thus reducing $\DD$ to $\RR$ [or $\LL$]. The \emph{Second Decomposition Theorem} (see \cite{Le89} Theorem 1.15) states that \emph{given any skew lattice $\mathbf S$, $\mathbf S/\RR$ and $\mathbf S/\LL$ are its respective maximal left and right-handed images, with $\mathbf S$ being isomorphic to the fibered product, $S/\RR \times_{S/\DD} S/\LL$, of both over their common maximal lattice image under the map $x \mapsto (\RR_{x}, \LL_{x})$}. 
All this is because every skew lattice is \emph{regular} in that for all $x,y,z\in S$ and all $x',x''\in \DD_{x}$ the following holds:

\begin{equation}\label{reg}\tag{2.5} 
x\vee y\vee x'\vee z\vee x = x\vee y\vee z\vee x \text{  and  } x\wedge y\wedge x'\wedge z\wedge x = x\wedge y\wedge z\wedge x.
\end{equation}

A skew lattice $\mathbf S$ is distributive (symmetric, cancellative, etc.) if and only if its left and right factors $\mathbf S/\RR$ and $\mathbf S/\LL$ are distributive (symmetric, cancellative, etc.). In general, $\mathbf S$ belongs to a variety $\VV$ of skew lattices if and only if both $\mathbf S/\RR$ and $\mathbf S/\LL$ do. (See also \cite{Ka05b} and \cite{Le96}, Section 1).

The natural preorder is defined on a skew lattice by 

\begin{equation}\label{pre}\tag{2.6} 
a \succeq b \Leftrightarrow a\vee b\vee a = a\text{  or, equivalently,  }b\wedge a\wedge b = b.   
\end{equation} 

\noindent Observe that $a\succeq b$ in $\mathbf S$ if and only if $\DD_{a} \geq \DD_{b}$ in the lattice $\mathbf S/\DD$ where$\DD_{a}$ and $\DD_{b}$ are the
respective $\DD$-classes of $a$ and $b$. Useful variants of \eqref{RH} and \eqref{LH} for the respective right and left-handed cases are as follows:

\begin{equation}\label{preR}\tag{2.7R} 
x \succeq x' \Rightarrow x\wedge y\wedge x' = y\wedge x' \text{  and  } x\vee y\vee x' = x\vee y; 
\end{equation}
\begin{equation}\label{preL}\tag{2.7L}
x \succeq x' \Rightarrow x'\wedge y\wedge x = x'\wedge y \text{  and  } x\vee y\vee x' = y\vee x'.
\end{equation}

\noindent We let $a \succ b$ denote $a \succeq b$ when $a \DD b$ does not hold.

\begin{lemma} 
For left-handed skew lattices, the following identities hold:
\begin{equation}\label{id1}\tag{2.8}
x\wedge (y\vee x) = x = (x\wedge y)\vee x. 
\end{equation}
\begin{equation}\label{id2}\tag{2.9}
(x\vee (y\wedge x))\wedge x = x\vee (y\wedge x) 
\end{equation}
\begin{equation}\label{id3}\tag{2.10}
(x\vee (y\wedge x))\wedge y = y\wedge x.
\end{equation}
\end{lemma}

\begin{proof} 
If $\mathbf S$ is a left-handed skew lattice, then $x\wedge (y\vee x) =_{\eqref{LH}} x\wedge (x\vee y\vee x) =_{\eqref{absidentities}} x$. 
Similarly, $(x\wedge y) \vee x = x$. 
As for \eqref{id2} observe that for \emph{all} skew lattices, $x\vee(y\wedge x)\vee x = x$, since $x \succeq y\wedge x$. Thus \eqref{id2} follows from \eqref{absequivalences}. \eqref{id3} follows from:

$\begin{array}{lcl}
(x \vee (y \wedge x)) \wedge y &=_{\eqref{id2},\eqref{LH}}& (x \vee (y \wedge x)) \wedge x \wedge y \wedge x \\
&=_{\eqref{id2}}& (x \vee (y \wedge x)) \wedge y \wedge x \\
&=_{\eqref{absidentities}}& y \wedge x.
\end{array}$
 
\end{proof}

The natural preorder $\succeq$ is refined by the \emph{natural partial order} which is defined on $\mathbf{S}$ by

\begin{equation}\label{poset}\tag{2.11}
x \geq y \leftrightarrow x\wedge y = y \wedge x = y \text{  or, equivalently,  }  x \vee y = y \vee x = x.
\end{equation}
All preorders and partial orders are assumed to be natural. Of course $x > y$ means $x \geq y$ but $x \neq y$.
Given $a \succeq b$, elements $a_b \in \DD_a$ and $b_a\in \DD_b$ exist such that $a \geq b_a$ and $a_b \geq b$. To see this just consider $a_b= b\vee a\vee b$ and $b_a= a\wedge b\wedge a$.


\section{Linear Distributivity}
\label{Linear Distributivity}

A skew lattice $\mathbf S$ is \emph{linearly distributive} if every subalgebra $\mathbf T$ that is totally preordered under $\succeq $ is distributive. 
Since totally preordered skew lattices are trivially symmetric, \emph{a skew lattice $\mathbf S$ is linearly distributive if and only if each totally preordered subalgebra $\mathbf T$ of $\mathbf S$ satisfies \eqref{GMD} or, equivalently, \eqref{GJD}}.

\begin{theorem} 
Linearly distributive skew lattices form a variety of skew lattices.
\end{theorem}

\begin{proof}
Consider the terms $x$, $y\wedge x\wedge y$ and $z\wedge y\wedge x\wedge y\wedge z$. 
Clearly $x \succeq y\wedge x\wedge y \succeq z\wedge y\wedge x\wedge y\wedge z$ holds for all skew lattices. 
Conversely given any instance $a \succeq b \succeq c$ in some skew lattice $\mathbf S$, the assignment $x \mapsto a$, $y \mapsto b$, $z \mapsto c$ will return this particular instance. 
Thus a characterizing set of identities for the class of all linearly distributive skew lattices is given by taking the basic identity
$$u\wedge (v \vee w)\wedge u = (u\wedge v\wedge u) \vee (u\wedge w\wedge u)$$
and forming all the identities possible in $x$, $y$, $z$ by making bijective assignments from the variables $\set{u, v, w}$ to the terms $\set{x, y\wedge x\wedge y, z\wedge y\wedge x\wedge y\wedge z}$. 
\end{proof}

In what follows, the following pair of lemmas will be useful.

\begin{lemma}\label{LRdist}
Left-handed skew lattices that satisfy \eqref{GMD} are characterized by:
\begin{equation}\label{GMDa}\tag{3.1L}
x\wedge y\wedge x=x\wedge y \text{  and  } x\wedge (y\vee z)=(x\wedge y)\vee (x\wedge z).   
\end{equation}
Dually, right-handed skew lattices that satisfy \eqref{GMD} are characterized by:
\begin{equation}\label{GMDb}\tag{3.1R}
x\wedge y\wedge x = y\wedge x \text{  and  } (y \vee z) \wedge x = (y \wedge x) \vee (z \wedge x). 
\end{equation}
\end{lemma}

\begin{lemma}\label{RLD} 
In a left-handed totally preordered skew lattice, if $a\wedge (b\vee c) \neq (a\wedge b) \vee (a\wedge c)$, then $a \succ b \succ c$. 
Thus a left-handed skew lattice $\mathbf S$ is linearly distributive if and only if
\begin{equation}\label{LLD}\tag{3.2L}
a\wedge ((b\wedge a) \vee (c\wedge b\wedge a)) = (a\wedge b) \vee (a\wedge c\wedge b) \text{  for all  } a, b, c \in S.  
\end{equation}
Dually a right-handed skew lattice $\mathbf S$ is linearly distributive if and only if
\begin{equation}\label{RLDeq}\tag{3.2R}
((a\wedge b\wedge c) \vee (a\wedge b))\wedge a = (b\wedge c\wedge a) \vee (b\wedge a) \text{  for all  } a, b, c \in S.  
\end{equation}
\end{lemma}

\begin{proof}
If say $b \succeq a$, then $a\wedge (b\vee c) = a$ and $(a\wedge b)\vee (a\wedge c) = a\vee (a\wedge c) = a$. 
If $c \succeq a$, then $a\wedge (b\vee c) = a$ again, and $(a\wedge b) \vee (a\wedge c) = (a\wedge b) \vee a = (a\wedge b\wedge a) \vee a = a$. 
Thus, inequality only occurs when $a\succeq b,c$. 
But even here, $a\succeq c\succeq b$ gives us $a\wedge (b\vee c)=a\wedge c$ and $(a\wedge c)\succeq (a\wedge b)$ so that $(a\wedge b) \vee (a\wedge c) = a\wedge c$ also. 
Thus, to completely avoid $a\wedge (b\vee c) = (a\wedge b) \vee (a\wedge c)$ we are only left with $a \succ b \succ c$. 
\end{proof}

Linear distributivity is also characterized succinctly by either of a dual pair of identities. We begin with an observation.

\begin{lemma}
Identities \eqref{GMD} and \eqref{GJD} are respectively equivalent to
\begin{equation}\label{GMDc}\tag{3.3}
x \wedge ((y \wedge x) \vee (z \wedge x)) = x \wedge (y \vee z) \wedge x = ((x \wedge y) \vee (x \wedge z)) \wedge x. 
\end{equation}
\begin{equation}\label{GMDd}\tag{3.4}
x \vee ((y \vee x) \wedge (z \vee x)) = x \vee (y \wedge z) \vee x = ((x \vee y) \wedge (x \vee z)) \vee x. 
\end{equation}
\end{lemma}

\begin{proof} 
Since $(y\wedge x)\vee (z\wedge x)\vee x=_{\eqref{absidentities}} x$, the skew lattice dualities \eqref{absequivalences} give us

\begin{equation}\label{equ}\tag{*}
((y \wedge x) \vee (z \wedge x)) \wedge x = (y \wedge x) \vee (z \wedge x). 
\end{equation}
Thus \eqref{GMD} implies,

$$
\begin{array}{lcl}
x\wedge ((y\wedge x)\vee (z\wedge x)) &=& x\wedge ((y\wedge x)\vee (z\wedge x))\wedge x \\
 &=& (x\wedge y\wedge x)\vee (x\wedge z\wedge x) \\
 &=& x\wedge (y\vee z)\wedge x. 
\end{array}
$$ 

\noindent Likewise, \eqref{GMD} implies $((x\wedge y)\vee (x\wedge y))\wedge x=x\wedge (y\vee z)\wedge x$. 
Conversely, if \eqref{GMDc} holds, then

$$
\begin{array}{lclcl}
x\wedge (y\vee z)\wedge x &=& x\wedge ((y\wedge x)\vee (z\wedge x)) &=_{\eqref{equ}}& x\wedge ((y\wedge x)\vee (z\wedge x))\wedge x \\
&=& ((x\wedge y\wedge x)\vee (x\wedge z\wedge x))\wedge x &=_{\eqref{equ}}& (x\wedge y\wedge x)\vee (x\wedge z\wedge x). 
\end{array}
$$

\end{proof}

\begin{corollary}
For all skew lattices, \eqref{GMD} and \eqref{GJD} imply respectively:
\begin{equation}\label{GMDe}\tag{3.5}
x\wedge ((y\wedge x)\vee (z\wedge x)) = ((x\wedge y)\vee (x\wedge z))\wedge x \text{  and}
\end{equation}
\begin{equation}\label{GMDf}\tag{3.6} 
x\vee ((y\vee x)\wedge (z\vee x)) = ((x\vee y)\wedge (x\vee z))\vee x
\end{equation}
\end{corollary}

\begin{theorem}\label{lindist} 
For all skew lattices, \eqref{GMDe} and \eqref{GMDf} are equivalent with a skew lattice satisfying either and hence both if and only if it is linearly distributive.
\end{theorem}

\begin{proof}
We begin with left-handed skew lattices. 
By Lemma \ref{RLD} we need only consider the case where $a \succ b \succ c$. 
The identity \eqref{GMDe} gives us the middle equality in the chain of equalities:
$$a\wedge (b\vee c) = a\wedge ((b\wedge a)\vee (c\wedge a)) = ((a\wedge b)\vee (a\wedge c))\wedge a = (a\wedge b)\vee (a\wedge c).$$
Thus $x\wedge (y \vee  z) = (x\wedge y) \vee (x\wedge z)$ holds in all totally preordered contexts in left-handed skew lattices satisfying \eqref{GMDe}. 
In such symmetrical contexts, the dual $(z \wedge y)\vee x = (z\wedge x) \vee (y\wedge x)$ also holds making the involved skew lattice linearly distributive. In dual fashion, right-handed skew lattices satisfying \eqref{GMDe} are also linearly distributive. 
Since any skew lattice $\mathbf S$ is embedded in the direct product $S/\RR \times S/\LL$, every skew lattice satisfying \eqref{GMDe} is linearly distributive. Conversely assume that  $\mathbf S$ is linear distributive. First, let $\mathbf S$ be left-handed. Then

\begin{center}
$
\begin{array}{lcl}
x \wedge ((y\wedge x) \vee (z\wedge x)) &=_{\eqref{LH}}& x \wedge ((z\wedge x) \vee (y\wedge x) \vee (z\wedge x)) \\
&=& (x \wedge ((z\wedge x) \vee (y\wedge x))) \vee (x \wedge (z\wedge x))  \\
&=_{\eqref{LH}}& (x \wedge ((y\wedge x) \vee (z\wedge x) \vee (y\wedge x))) \vee (x \wedge (z\wedge x)) \\
&=& (x \wedge ((y\wedge x) \vee (z\wedge x))) \vee (x \wedge (y\wedge x)) \vee (x \wedge (z\wedge x)) \\
&=& (x \wedge  (y\wedge x)) \vee (x \wedge (z\wedge x)) \\
&=_{\eqref{LH}}& (x\wedge y)\vee (x\wedge z) \\
&=& ((x\wedge y)\vee (x\wedge z))\wedge x.
\end{array}
$
\end{center}

\noindent Here the second and fourth equalities follow from linear distributivity. 
The fifth equality is again left-handedness upon observing that $x \wedge ((y\wedge x) \vee (z\wedge x))$ and $(x \wedge (y\wedge x)) \vee (x \wedge (z\wedge x))$ are $\LL$-related (look at $S/\DD = S/\LL$). 
The final equality follows from the fact that $x \geq (x\wedge y) \vee (x\wedge z)$ in the left-handed case. 
Thus \eqref{GMDe} holds. Similarly \eqref{GMDe} holds for linearly distributive, right-handed skew lattices. 
Again the embedding $S \rightarrow S/\RR \times S/\LL$ guarantees that all linearly distributive skew lattices satisfy \eqref{GMDe}. 
Thus linear distributivity is characterized by \eqref{GMDe}. The dual argument gives a characterization by \eqref{GMDf}. 
\end{proof}

\begin{corollary}
For left- and right-handed skew lattices, \eqref{GMDe} reduces respectively to
\begin{equation}\label{GMDg}\tag{3.5L}
x\wedge ((y\wedge x)\vee (z\wedge x)) = (x\wedge y)\vee (x\wedge z)\text{  and}
\end{equation}
\begin{equation}\label{GMDh}\tag{3.5R} 
((x\vee y)\wedge (x\vee z))\vee x = (y\wedge x)\vee (z\wedge x)
\end{equation}
\end{corollary}


\section{Midpoints and Distributive Skew Chains}
\label{Midpoints and Distributive Skew Chains}

A skew lattice is linearly distributive if and only if each skew chain of $\DD$-classes in it is distributive. 
In this section we characterize distributive skew chains in terms of the natural partial order. 
Given a skew chain $A>B>C$ where $A$, $B$ and $C$ are $\DD$-classes, with $a \in A$, $c \in C$ such that $a > c$, any element $b \in B$ such that $a > b > c$ is called a \emph{midpoint} in $B$ of $a$ and $c$. We begin with several straightforward assertions.

\begin{lemma}\label{midpoint}
Given a skew chain $A > B > C$, with $a > c$ for all $a \in A$ and $c \in C$:
\begin{itemize}
\item[(i)] For all $b \in B$, $a\wedge (c\vee b\vee c)\wedge a$ and $c\vee (a\wedge b\wedge a)\vee c$ are midpoints in $B$ of $a$ and $c$.
\item[(ii)] If $b$ in $B$ is a midpoint of $a$ and $c$, then both midpoints in (i) reduce to $b$.
\item[(iii)] When $A > B > C$ is a distributive skew chain, both midpoints in (i) agree:
\begin{equation}\label{distineq}\tag{4.1} a > a\wedge (c\vee b\vee c)\wedge a = c\vee (a\wedge b\wedge a)\vee c > c. \end{equation}
\end{itemize}
\end{lemma}

Midpoints provide a key to determining the effects of \eqref{GMD} and \eqref{GJD} in this context. 
To proceed further, we recall several concepts. Given a skew chain $A > B > C$, an \emph{$A$-coset in $B$} is any subset of $B$ of the form $A\wedge b\wedge A = \set{a\wedge b\wedge a' \mid a, a'\in A}$ for some $b$ in $B$. 
Given two $A$-cosets in $B$, they are either identical or else disjoint. 
Since $b \in A\wedge b\wedge A$ for all $b$ in $B$, the $A$-cosets in $B$ form a partition of $B$. 
Dually a \emph{$B$-coset in $A$} is a subset of $A$ of the form $B\vee a\vee B = \set{b\vee a\vee b' \mid b, b'\in B}$ for some $a$ in $A$. 
Again, the $B$-cosets in $A$ form a partition of $A$. Given a $B$-coset $X \subseteq A$ and an $A$-coset $Y \subseteq B$, a \emph{coset bijection} $\varphi: X\rightarrow Y$ is given by $\varphi(a) = b$ if $a \in X$, $b\in Y$ and $a>b$. 
Alternatively, $\varphi(a) = a\wedge b\wedge a$ and, dually, $\varphi^{-1}(b) = b\vee a\vee b$ for all $a \in X$ and all $b \in Y$ . 
Cosets are rectangular subalgebras in their $\DD$-classes and all coset bijections are isomorphisms. Thus all $A$-cosets in $B$ and all $B$-cosets in $A$ have a common size, denoted by $\omega[A,B]$. If $a, a' \in A$ lie in a common $B$-coset, this is denoted by $a -_{B} a'$; likewise $b -_{A} b'$ if $b$ and $b'$ lie in a common $A$-coset in $B$. This is illustrated in the partial configuration below where the dashed lines indicate $>$ between $a$'s and $b$'s while the full lines represent $\DD$-related elements.  

\begin{center}
\begin{tikzpicture}[scale=1.3]

  \node (A) at (0,1){$A:$} ;
  \node (B) at (0,0){$B:$} ;
  \node (a1) at (1,1){$a_{1}$} ;
  \node (a2) at (2,1){$a_{2}$};
  \node (a3) at (3,1){$a_{3}$}  ;
  \node (a4) at (4,1){$a_{4}$} ;
    \node (b1) at (2,0){$b_{1}$} ;
  \node (b2) at (3,0){$b_{2}$};
  
    \draw[dotted] (a1) -- (b1);
    \draw[dotted] (a2) -- (b2);
    \draw[dotted] (a3) -- (b1);
    \draw[dotted] (a4) -- (b2);
    
    \draw (a1) -- (a2) node[pos=.5,below] {$B$};
    \draw (a3) -- (a4) node[pos=.5,below] {$B$};
    \draw (b1) -- (b2) node[pos=.5,below] {$A$};
    
\end{tikzpicture}
\end{center}

Binary outcomes between elements in $A$ and $B$ are given by, e.g., $a\wedge b =\varphi (a)\wedge b$ in $B$ and $a\vee b =a\vee \varphi^{-1}(b)$ in $A$ using the relevant coset bijection $\varphi: B\vee a\vee B \rightarrow A\wedge b\wedge A$. (For more details see \cite{Le93} and \cite{Le96} or remarks in \cite{CAT}.)

Similarly there are $A$-cosets in $C$, $C$-cosets in $A$, $B$-cosets in $C$ and $C$-cosets in $B$. 
The $C$-coset decomposition of $A$ refines the $B$-coset decomposition of $A$; similarly $B$-cosets in $C$ are refined by $A$-cosets in $C$. 
Our interest is in the middle class $B$ of the skew chain. Elements $b$ and $b'$ in $B$ are \emph{$AC$-connected} if a finite sequence $b = b_{0}, b_{1}, b_{2}, \dots , b_{n} =b'$ exists in $B$ such that $b_{i} -_{A} b_{i+1}$ or $b_{i} -_{C} b_{i+1}$ for all $i\leq n-1$. 
A maximally $AC$-connected subset of $B$ is an \emph{$AC$-component} of $B$ (or just \emph{component} if the context is clear). 
$B$ is a disjoint union of all its $AC$-components and every $AC$-component in $B'$ is the disjoint union of all $A$-cosets in $B$ that are subsets of $B'$ and the disjoint union of all $C$-cosets in $B$ that are subsets of $B'$, as well as the disjoint union of all the $AC$-cosets in $B'$.  
$AC$-connectedness is a congruence relation on $B$. Its congruence classes, the components, are thus subalgebras of $B$. 
Given a component $B'$ of $B$, a sub-skew chain is given by $A > B' > C$. 
Since $a\wedge (c\vee b\vee c)\wedge a$ is the same for all $b$ in a common $C$-coset and $c\vee (a\wedge b\wedge a)\vee c$ is the same for all $b$ in a common $A$-coset, we can extend Lemma \ref{midpoint} as follows:

\begin{lemma}
Given a distributive skew chain $A > B > C$, for any pair $a > c$ where $a \in A$ and $c \in C$, each $AC$-component $B'$ of $B$ contains a unique midpoint $b$ of $a$ and $c$.  
\end{lemma}

Given cosets $X \subseteq A$ and $Y \subseteq B$ as above, a coset bijection $\varphi:X\rightarrow Y$ can be viewed as a partial bijection between the involved $\DD$-classes, $\varphi:A\rightarrow B$. Recall that a skew lattice $\mathbf S$ is \emph{categorical} if for all skew chains $A > B > C$ of $\DD$-classes in $\mathbf S$, nonempty composites $\psi\circ \varphi$ of coset bijections $\varphi$ from $A$ to $B$ and $\psi$ from $B$ to $C$ are coset bijections from $A$ to $C$. In this case, adjoining empty partial bijections to account for empty compositions and identity bijections on $\DD$-classes, one obtains a category with $\DD$-classes for objects, coset bijections for morphisms, and the composition of partial functions for composition (see \cite{CAT}, \cite{JPC11} or \cite{JPC12} for more details). Clearly, a skew chain $A>B>C$ is categorical if and only if $A > B' > C$ is categorical for each component $B'$. Categorical skew lattices form a variety (see \cite{Le93}, Theorem 3.16). We also have:

\begin{theorem}[\cite{CAT}, Theorem 2.3] 
A skew lattice $\mathbf S$ is categorical if and only if for all $x,y, z \in S$.
\begin{equation}\label{cateq}\tag{4.2} 
x\geq y \succeq z \text{  implies  } x \wedge (z \vee y \vee z) \wedge x = (x \wedge z \wedge x) \vee y \vee (x \wedge z \wedge x) 
\end{equation}
\end{theorem}

Thus linearly distributive skew lattices are categorical. The converse, however, does not hold (see Example \ref{nlindist} below). 
It does hold, however, for \emph{strictly categorical} skew lattices where, in addition, for every chain of $\DD$-classes $A > B > C$ each $A$-coset in $B$ has nonempty intersection with each $C$-coset in $B$, making of $B$ a single $AC$-component. 
Strictly categorical skew lattices form a variety (see \cite{CAT}, Corollary 4.3). This class includes:

\begin{itemize}
\item[(a)] \emph{Normal} skew lattices characterized by the condition $x\wedge y\wedge z\wedge w = x\wedge z\wedge y\wedge w$, or equivalently, every subset $[e]\downarrow = \set{x \in S\mid e \geq x}$ is a sublattice (see \cite{Le92}). Skew Boolean algebras are normal as skew lattices.
\item[(b)] \emph{Primitive} skew lattices consisting of two $\DD$-classes, $A > B$, and all skew lattices in the subvariety generated from this class of skew lattices.
\end{itemize}

\begin{theorem}[\cite{CAT}, Theorem 4.2]\label{strict} 
The following conditions on a skew lattice $\mathbf S$ are equivalent:
\begin{itemize}
\item[(i)] $\mathbf S$ is strictly categorical.
\item[(ii)] Given both $a > b > c$ and $a > b' > c$ in $S$ with $b \DD b'$, $b = b'$ follows.
\item[(iii)] Given $a > b$ in $S$, the subalgebra $[a, b] = \set{x\in S\mid a \geq x \geq b}$ is a sublattice.
\item[(iv)] $\mathbf S$ is categorical and given skew chain $A > B > C$ in $\mathbf S$, for each coset bijection $\chi: A \rightarrow C$ unique coset bijections $\varphi: A \rightarrow B$ and $\psi: B \rightarrow C$ exist such that $\chi=\psi \circ \varphi$. 
\end{itemize}
\end{theorem}

Returning to distributive skew chains we have the following:

\begin{lemma}\label{distlema}
A left-handed, categorical skew chain $\mathbf S$ is distributive if and only if $a\wedge (b\vee c) = (a\wedge b) \vee (a\wedge c)$ for all $a \succ b \succ c$ such that $a > c$, in which case the identity reduces to $a\wedge (b\vee c) = (a\wedge b)\vee c$. 
Dually, a right-handed categorical skew chain $\mathbf S$ is distributive if and only if $(c\vee b)\wedge a = (c\wedge a)\vee (b\wedge a)$ for all $a \succ b \succ c$ such that $a > c$, in which case the identity reduces to $(c\vee b)\wedge a = c\vee (b\wedge a)$. (Note that these identities are the left and right-handed cases of \eqref{distineq} above.)
\end{lemma}

\begin{proof}
Given $a\succ b\succ c$ with respective $\DD$-classes $A>B>C$, let $c'=a\wedge c$. Then $a>c'$ and $(a\wedge b) \vee (a\wedge c) = (a\wedge b) \vee c'$. Next, since $c$ and $c'$ lie in the same $A$-coset in $C$ and $\mathbf S$ is categorical, both $b\vee c$ and $b\vee c'$ lie in the same $A$-coset in $B$ so that $a\wedge (b\vee c) = a\wedge (b\vee c')$. Hence $a\wedge (b\vee c) = (a\wedge b) \vee (a\wedge c)$ if and only if $a\wedge (b \vee c') = (a\wedge b) \vee (a\wedge c')$ where $a \succ b \succ c'$, $a > c'$ with the latter expression reducing to $(a\wedge b) \vee c'$ as stated. 
The lemma follows from Lemma \ref{LRdist} and left-right duality. 
\end{proof}

\begin{theorem}\label{disteq} 
Given a skew chain $A > B > C$, the following condition are equivalent:
\begin{itemize}
\item[(i)] $A > B > C$ is distributive.
\item[(ii)] For all $a \in A$, $b \in B$ and $c \in C$ with $a > c$, $$a\wedge (c\vee b\vee c)\wedge a = c\vee (a\wedge b\wedge a)\vee c.$$
\item[(iii)] Given $a \in A$ and $c \in C$ with $a > c$, each component $B'$ of $B$ contains a unique midpoint $b$ of $a$ and $c$. 
\item[(iv)] For each component $B'$ of $B$, $A>B'>C$ is strictly categorical.
\end{itemize}

When these conditions hold, each coset bijection $\chi: A \rightarrow C$ uniquely factors through each component $B'$ of $B$ in that unique coset bijections $\varphi: A \rightarrow B'$ and $\psi: B' \rightarrow C$ exist such that $\chi = \psi \circ \varphi$ under the usual composition of partial bijections.
\end{theorem}

\begin{proof}
(i) clearly implies (ii). 
Given $a > c$ in (ii), for each element $x$ in $B$, both $b_{1} = a\wedge (c\vee x\vee c)\wedge a$ and $b_{2} = c\vee (a\wedge x\wedge a)\vee c$ are midpoints of $a$ and $c$ in $B$. 
Replacing $x$ by any element in its $C$-coset, does not change the $b_{1}$-outcome. 
Likewise, replacing $x$ by any element in its $A$-coset, does not change the $b_{2}$-outcome. 
Hence (ii) is equivalent to asserting that given $a > c$ fixed, for all $x$ in a common $AC$-component $B'$ of $B$, both $a\wedge (c\vee x\vee c)\wedge a$ and $c\vee (a\wedge x\wedge a)\vee c$ produce the same output $b$ in $B'$ such that $a > b > c$. 
Conversely, for any $b$ in $B'$ such that $a > b > c$ we must have $a\wedge (c\vee b\vee c)\wedge a = b = c\vee (a\wedge b\wedge a)\vee c$. 
Thus (ii) and (iii) are equivalent. Their equivalence with (iv) follows from Theorem \ref{strict} above. Given (ii) Ð (iv), (iv) forces $A > B > C$ to categorical, since for each component $B'$ in $B$, $A > B' > C$ is categorical. Denoting the skew chain by $\mathbf S$, (ii) forces $\mathbf S/\RR$ and $\mathbf S/\LL$ to be distributive by Lemma \ref{distlema} and thus $S \subseteq S/\RR \times S/\LL$ to be distributive. 
In the light of Theorem \ref{strict}, the final comment is clear. 
\end{proof}

\begin{corollary}
A strictly categorical skew lattice is linearly distributive. 
\end{corollary}

Given $a > c$ as above, their midpoint $b$ in the component $B'$ depends on the interplay of the $A$-cosets and $C$-cosets within $B'$. 
Indeed, given any $a \in A$, the set of \emph{images} of $a$ in $B'$, is the set $a\wedge B'\wedge a = \set{a\wedge b\wedge a\mid b \in B'} = \set{b \in B'\mid a > b}$. This set parameterizes the $A$-cosets in $B'$ since each possesses exactly one $b$ such that $a > b$. 
Likewise, for each $c \in C$ the image set $c\vee B'\vee c = \set{c\vee b\vee c\mid b \in B'} = \set{b \in B'\mid b > c}$ parameterizes all cosets of $C$ in $B'$ (see \cite{Le93}, Section 1). Both images sets are orthogonal in $B'$ in the following sense: for any $a \in A$, all images of $a$ in $B'$ lie in a unique $C$-coset in $B'$. Likewise for any $c \in C$, all images of $c$ in $B'$ lie in a unique $A$-coset in $B'$. Finally, given $a>c$ with $a\in A$ and $c\in C$, their unique midpoint $b \in B'$ lies jointly in the $C$-coset in $B'$ containing all images of $a$ in $B'$ and in the $A$-coset in $B'$ containing all images of $c$ in $B'$. (See \cite{CAT}, Theorem 4.1.) Of course, every $b$ in $B'$ is the midpoint of some pair $a>c$. For a fixed pair $a>c$, the set $\mu(a, c)$ of all midpoints in $B$ is a rectangular subalgebra that forms a natural set of parameters for the family of all $AC$-components in $B$:
just let $b$ in $\mu(a, c)$ correspond to the component $B'$ containing $b$.
In the following partition diagram, the $A$-coset of $b$ contains all images ($\bullet$'s) of $c$ in $B'$, while the $C$-coset of $b$ has all images ($\star$'s) of $a$ in $B'$. The element $b$ is the unique image of both $a$ and $c$.

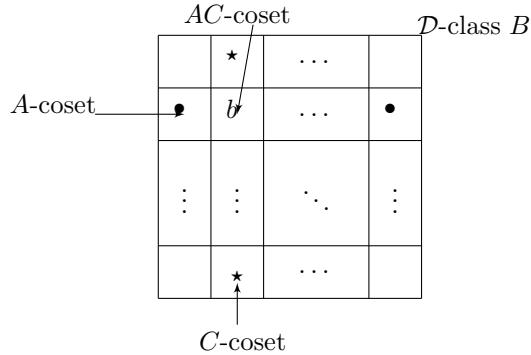
\begin{center}
\begin{tikzpicture}[scale=0.7]

\draw (-2,2) -- (3,2);
\draw (-2,1) -- (3,1);
\draw (-2,0) -- (3,0);
\draw (-2,-2) -- (3,-2);
\draw (-2,-3) -- (3,-3);
\draw (-2,-3) -- (-2,2);
\draw (-1,-3) -- (-1,2);
\draw (0,-3) -- (0,2);
\draw (2,-3) -- (2,2);
\draw (3,-3) -- (3,2);

\node at (4,2.2){$\DD$-class $B$};
\node at (-.5,2.4){$AC$-coset};
\node (x) at (1,1.5){$\dots$};
\node (2x) at (1,0.5){$\dots$};
\node (3x) at (1,-2.5){$\dots$};
\node (4x) at (1,-1){$\ddots$};
\node (5x) at (-1.5,-1){$\vdots$};
\node (6x) at (-0.5,-1){$\vdots$};
\node (7x) at (2.5,-1){$\vdots$};
\node (A) at (-4,0.7){$A$-coset};
\node (C) at (-0.4,-3.8){$C$-coset};
\node (bb) at (-0.6,0.6){$b$};
\node (bu1) at (-1.6,0.6){$\bullet$};
\node (st1) at (-0.6,1.6){$\star$};
\node (st2) at (-0.5,-2.6){$\star$};
\node (bu2) at (2.4,0.6){$\bullet$};

\draw[arrows=-latex'] (-.2,2.2) -- (-0.5,0.5);
\draw[arrows=-latex'] (-3.2,0.5) -- (-1.5,0.5);
\draw[arrows=-latex'] (-0.5,-3.5) -- (-0.5,-2.7);

\end{tikzpicture}
\end{center}

\begin{example}\label{nlindist}
Using Mace4 \cite{prover}, two minimal 12-element categorical skew chains have been found that are not linearly distributive, one left-handed and the other its right- handed dual. Their common Hasse diagram follows where $b_{i} -_{C} d_{j}$ iff $i + j = 0$ (mod 4).

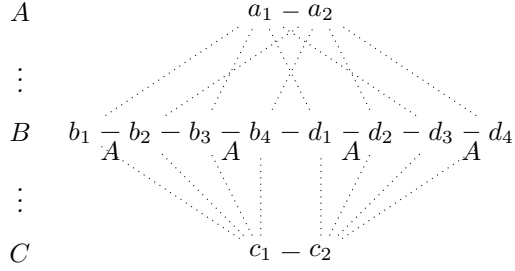
\begin{center}
\begin{tikzpicture}[scale=0.8]

  \node (A) at (0,2){$A$} ;
  \node (x) at (0,1){$\vdots$} ;
  \node (B) at (0,0){$B$} ;
  \node (xx) at (0,-1){$\vdots$} ;
  \node (C) at (0,-2){$C$} ;
  
  \node (a1) at (4,2){$a_{1}$} ;
  \node (a2) at (5,2){$a_{2}$};
  \node (c1) at (4,-2){$c_{1}$}  ;
  \node (c2) at (5,-2){$c_{2}$} ;
  \node (b1) at (1,0){$b_{1}$} ;
  \node (b2) at (2,0){$b_{2}$};
  \node (b3) at (3,0){$b_{3}$} ;
  \node (b4) at (4,0){$b_{4}$};
  \node (d1) at (5,0){$d_{1}$} ;
  \node (d2) at (6,0){$d_{2}$};
  \node (d3) at (7,0){$d_{3}$} ;
  \node (d4) at (8,0){$d_{4}$};
  
    \draw[dotted] (a1) -- (b1);
    \draw[dotted] (a2) -- (b2);
    \draw[dotted] (a1) -- (b3);
    \draw[dotted] (a2) -- (b4);
    \draw[dotted] (a1) -- (d1);
    \draw[dotted] (a2) -- (d2);
    \draw[dotted] (a1) -- (d3);
    \draw[dotted] (a2) -- (d4);
    \draw[dotted] (c1) -- (b1);
    \draw[dotted] (c1) -- (b2);
    \draw[dotted] (c1) -- (b3);
    \draw[dotted] (c1) -- (b4);
    \draw[dotted] (c2) -- (d1);
    \draw[dotted] (c2) -- (d2);
    \draw[dotted] (c2) -- (d3);
    \draw[dotted] (c2) -- (d4);
        
    \draw (b1) -- (b2) node[pos=.5,below] {$A$};
    \draw (b2) -- (b3) node[pos=.5,below] {};
    \draw (b3) -- (b4) node[pos=.5,below] {$A$};
    \draw (b4) -- (d1) node[pos=.5,below] {};
    \draw (d1) -- (d2) node[pos=.5,below] {$A$};
    \draw (d2) -- (d3) node[pos=.5,below] {};
    \draw (d3) -- (d4) node[pos=.5,below] {$A$};
    \draw (a1) -- (a2) node[pos=.5,below] {};
    \draw (c1) -- (c2) node[pos=.5,below] {};
    
\end{tikzpicture}
\end{center}

\noindent In both cases, $a_{1} >b_{odd},d_{odd}$ and $a_{2} >b_{even},d_{even}$, all $b_{i} >c_{1}$, all $d_{i} >c_{2}$, and $a_{1},a_{2} >$ both $c_{1}, c_{2}$. (thus both skew chains are categorical since all cosets involving just $A$ and $C$ are trivial). 
We denote the left-handed skew lattice thus determined by $\mathbf U_{2}$ and its right-handed dual by $\mathbf V_{2}$. 
Both $\mathbf U_{2}$ and $\mathbf V_{2}$ are not distributive. 
Indeed, given the coset structure on $B$, we get $a_{1}\wedge (b_{2}\vee c_{2}) = a_{1}\wedge d_{2} = d_{1}$, while $(a_{1}\wedge b_{2}) \vee (a_{1}\wedge c_{2}) = b_{1} \vee c_{2} = d_{3} \neq d_{1}$ in $\mathbf U_{2}$. $\mathbf V_{2}$ is handled similarly. 
Note that in both $\mathbf U_{2}$ and $\mathbf V_{2}$, $B$ is an $AC$- connected, but $a_{1} > b_{1}$, $b_{3} > c_{1}$, and also $a_{2} > b_{2}$, $b_{4} > c_{1}$, etc.
\end{example}

\noindent (Strictly) categorical skew lattices were studied in \cite{CAT}. 
A number of lovely counting results for finite strictly categorical skew chains may be found in \cite{JPC11} or \cite{JPC12}.


\section{From Linear Distributivity to Distributive Skew Lattices}
\label{From Linear Distributivity to Distributive Skew Lattices}

One may ask: \emph{Does linearly distributive plus quasi-distributive imply distributive?} 
In general the answer is no. It is however yes in two special cases. 
In \cite{Le92} it was shown that a normal skew lattice is distributive if and only it is quasi-distributive. 
This result can be extended to strictly categorical skew lattices. 
But first recall from \cite{Ka11c} that a skew lattice $\mathbf S$ is simply cancellative if for all $x, y, z \in S$,
\begin{equation}\label{simpcanc}\tag{5.1}
x\vee z\vee x = y\vee z\vee y \text{  and  } x\wedge z\wedge x = y\wedge z\wedge yÊ\text{  imply  } x = y. 
\end{equation}
Cancellative skew lattices are simply cancellative, and simply cancellative skew lattices in turn are quasi-distributive since \eqref{simpcanc} rules out $\mathbf M_{3}$ and $\mathbf N_{5}$ as subalgebras.

\begin{theorem}\label{strcat} 
Strictly categorical, quasi-distributive skew lattices are both distributive and simply cancellative. 
They are cancellative precisely when they are also symmetric.
\end{theorem}

\begin{proof}
In general, $a > a\wedge (b\vee c)\wedge a$ and $a > (a\wedge b\wedge a) \vee (a\wedge c\wedge a)$ both hold. 
In turn, so do both $a\wedge c\wedge b\wedge a < a\wedge (b\vee c)\wedge a$ and $a\wedge c\wedge b\wedge a < (a\wedge b\wedge a) \vee (a\wedge c\wedge a)$. Indeed, applying regularity and absorption we have, e.g.,

$$
\begin{array}{lcl}
(a\wedge c\wedge b\wedge a)\wedge [a\wedge (b\vee c)\wedge a] &=_{\eqref{reg}}& a\wedge c\wedge b\wedge (b\vee c)\wedge a \\
&=_{\eqref{absidentities}}& a\wedge c\wedge b\wedge a  \\

\\
&\text{and}& \\
\\

(a\wedge c\wedge b\wedge a)\wedge [(a\wedge b\wedge a) \vee (a\wedge c\wedge a)] &=_{\eqref{reg}}& a\wedge c\wedge a\wedge b\wedge a\wedge [(a\wedge b\wedge a) \vee (a\wedge c\wedge a)] \\
&=_{\eqref{absidentities}}& a\wedge c\wedge a\wedge b\wedge a =_{\eqref{reg}} a\wedge c\wedge b\wedge a
\end{array}
$$

In any quasi-distributive skew lattice $\mathbf S$, $a\wedge (b\vee c)\wedge a \Drel (a\wedge b\wedge a) \vee (a\wedge c\wedge a)$. 
Thus if $\mathbf S$ is quasi-distributive and strictly categorical, Theorem \ref{strict} implies that \eqref{GMD} and dually \eqref{GJD} must hold.

Let $x,y,z\in S$ be such that $x\vee z\vee x=y\vee z\vee y$ and $x\wedge z\wedge x=y\wedge z\wedge y$. 
If $\mathbf S/\DD$ is distributive, then $x$ and $y$ share a common image in $\mathbf S/\DD$, placing them in the same $\DD$-class in $\mathbf S$. 
But also $x\wedge z\wedge x \leq$ both $x, y \leq x\vee z\vee x$. If $\mathbf S$ is strictly categorical, then Theorem \ref{strict} gives $x=y$. 
\end{proof}

\begin{corollary}
A strictly categorical skew lattice is distributive if and only if no subalgebra is a copy of $\mathbf M_{3}$ or $\mathbf N_{5}$
\end{corollary}

Given Theorems \ref{disteq} and \ref{strcat} one might expect linearly distributive, quasi-distributive skew lattices to be distributive. 
Mace4, however, has produced four minimal counterexamples. 
They turn out to be Spinks' minimal 9-element examples of skew lattices for which exactly one of \eqref{GMD} or \eqref{GJD} hold. (See \cite{Sp00} and Example \ref{nondistrib} below.) 
Since \eqref{GMD} and \eqref{GJD} are equivalent for symmetric skew lattices and skew chains are always symmetric, these examples are linearly distributive, so that appropriate products of them are both linearly distributive and quasi-distributive, but satisfy neither \eqref{GMD} nor \eqref{GJD}.

Spinks' examples are necessarily non-symmetric. This leads one to ask: Do linear distributivity and quasi-distributivity jointly imply distributivity for symmetric skew lattices? 
\smallskip

Before showing this to be the case, we first consider the broader problem of deciding which linearly distributive, quasi-distributive skew lattices are distributive. 
To see what else is required, we begin with a property common to all skew lattices.

Given $\DD$-classes $A$ and $B$, their meet class $M$, and an element $m \in M$, then $a\wedge b\wedge a = a\wedge b'\wedge a$ for all $a\in A$ and all $b,b' \in Im(m\mid B)$, the set of all images of $m$ in $B$.

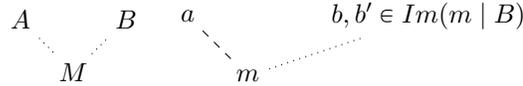
\begin{center}
$\begin{array}{cc}

\begin{tikzpicture}[scale=.7]

  \node (a) at (-1,1){$A$} ;
  \node (b) at (1,1){$B$} ;
  \node (m) at (0,0){$M$};
  
    \draw[dotted] (a) -- (m) -- (b);

\end{tikzpicture}

&

\begin{tikzpicture}[scale=.8]

  \node (a) at (-1,1){$a$} ;
  \node (b) at (3,1){$b,b' \in Im(m\mid B)$} ;
  \node (m) at (0,0){$m$};
  
    \draw[dashed] (a) -- (m) ;
    \draw[dotted] (b) -- (m) ;

\end{tikzpicture}

\end{array}$
\end{center}

\noindent Indeed if $a'\in  Im(m\mid A)$ so that $a'\wedge b = m = b\wedge a'$ for all $b \in Im(m\mid B)$, then regularity implies $a\wedge b\wedge a = a\wedge a'\wedge b\wedge a'\wedge a = a\wedge m\wedge a$, and this occurs for all $b \in Im(m\mid B)$. 
More generally we have:
\smallskip

The \emph{Meet-class Condition} (MCC): Given $\DD$-classes $A$ and $B$, their meet-class $M$, and elements $a, a' \in A$ and $b, b' \in B$, then $a\wedge b\wedge a = a\wedge b'\wedge a$ if $b$ and $b'$ share a common image in $M$. 
Likewise $b\wedge a\wedge b = b\wedge a'\wedge b$ if $a$ and $a'$ share a common image in $B$. 
Finally, all four outcomes coincide when $a$, $a'$, $b$ and $b'$ all share a common image in $M$.

Dualizing, one has the \emph{Join-class Condition} (JCC), with all skew lattices having both properties. 
 
\smallskip

Not all skew lattices, however, have their following extensions:

\smallskip
 
The \emph{Extended Meet-class Condition} (EMCC). Given $\DD$-classes $A$ and $B$, their meet class $M$, an element $a$ in $A$ and elements $d$ and $d'$ in a $\DD$-class $D$ lying above $B$, then $a\wedge d\wedge a = a\wedge d'\wedge a$ if $d$ and $d'$ share a common image in $M$ and a common $B$-coset in $D$.
 
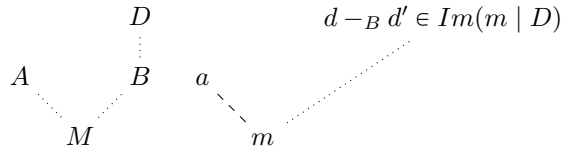
\begin{center}
$\begin{array}{cc}

\begin{tikzpicture}[scale=.8]

  \node (a) at (-1,1){$A$} ;
  \node (b) at (1,1){$B$} ;
  \node (m) at (0,0){$M$};
  \node (d) at (1,2){$D$};
  
    \draw[dotted] (a) -- (m) -- (b) -- (d);

\end{tikzpicture}

&

\begin{tikzpicture}[scale=.8]

  \node (a) at (-1,1){$a$} ;
  \node (b) at (3,2){$d -_{B} d' \in Im(m\mid D)$} ;
  \node (m) at (0,0){$m$};
  
    \draw[dashed] (a) -- (m) ;
    \draw[dotted] (b) -- (m) ;

\end{tikzpicture}

\end{array}$
\end{center}

Dually, there is the \emph{Extended Join-class Condition} (EJCC).
\smallskip

An equivalent formulation of the (E)MCC requires a broader way to describe cosets. 
Given $\DD$-classes $A$ and $B$, set $A\wedge b\wedge A = \set{a\wedge b\wedge a'\mid a,a'\in A}$ for any $b\in B$. 
If $A \geq B$, then $A\wedge b\wedge A$ is just a typical $A$-coset in $B$. 
If $B \geq A$, then $A\wedge b\wedge A = A$, the unique $A$-coset in itself. 
In general, setting $M = A\wedge B$, regularity \eqref{reg} and other basic facts imply:

\begin{itemize} 
\item[i)] Given $b \geq m$ where $b \in B$ and $m \in M$, $A\wedge b\wedge A = A\wedge m\wedge A$, an $A$-coset in $M$; conversely every coset $A\wedge m\wedge A$ of $A$ in $M$ is just $A\wedge b\wedge A$ for some $b$ in $B$.
\item[ii)] For all $b$, $b'$ in $B$, if $A\wedge b\wedge A = A\wedge b'\wedge A$, then $a\wedge b\wedge a = a\wedge b'\wedge a$ for all $a$ in $A$, with $A\wedge b\wedge A$ being just $\set{a\wedge b\wedge a\mid a\in A}$.
\end{itemize}

Indeed, given $b \geq m$, pick $a_m \in A$ so that $a_m \geq m$. 
Then $m = a_m\wedge b$ so that $a\wedge b\wedge a'= a\wedge a_m\wedge b\wedge a' = a\wedge m\wedge a'$ by \eqref{reg}. 
Conversely each $m\in M$ factors as some $a_m\wedge b$; thus $a\wedge m\wedge a' = a\wedge a_m\wedge b\wedge a' = a\wedge b\wedge a'$ by \eqref{reg} and (i) follows. 
Note that $A\wedge b\wedge A = \set{a\wedge b\wedge a\mid a\in A}$ since \eqref{reg} gives $a\wedge b\wedge a' = a\wedge a'\wedge b\wedge a\wedge a'$. 
The remainder of (ii) also follows from \eqref{reg}.

The MCC is thus equivalent to: given $A$, $B$ and $M$ as above, $b, b' \geq  m$ for $m\in M$ and $b, b'\in B$ implies $A\wedge b\wedge A = A\wedge b'\wedge A$ as cosets of $A$ in $M$. 
The EMCC is likewise equivalent to: given also $d, d' \geq m$ for $m\in M$ and $d -_B d'\in D \geq B$, $A\wedge d\wedge A = A\wedge d'\wedge A$ as cosets of $A$ in $A\wedge D$. 
Dual remarks apply to the (extended) join-class condition.

\begin{theorem}\label{EMCC}
A skew lattice $\mathbf S$ has the EMCC property if and only if it satisfies 

\begin{equation}\label{dist3}\tag{5.3}
x\wedge ((y\wedge x\wedge y)\vee z\vee y\vee z\vee( y\wedge x\wedge y))\wedge x = x\wedge (y\vee z\vee y)\wedge x.
\end{equation}

$\mathbf S$ has the EJCC property if and only if it satisfies the dual of \eqref{dist3}. 
Skew lattices having the EMCC property [or the EJMC property] thus form a subvariety. 
Finally, $\wedge$-distributivity, given by \eqref{GMD}, implies EMCC, while $\vee$-distributivity, given by \eqref{GJD}, implies EJCC. 
\end{theorem}

\begin{proof}
Setting $a = x$, $m = y\wedge x\wedge y$, $B = \DD_{y}$, $D = \DD_{d}= \DD_{d'}$ where $d = y\vee z\vee y$ and $d' = (y\wedge x\wedge y)\vee z\vee y\vee z\vee (y\wedge x\wedge y)$, the EMCC gives \eqref{dist3}. 

Conversely, given $a$, $m$, $d$ and $d'$ satisfying the requirements of \eqref{dist3}, first pick $b$ in $B$ so that $m < b < d$ and pick $a'$ in $A$ so that $a'\wedge b = b\wedge a' = m$. 
Next let $d' = b'\vee d\vee b'$ for some $b'$ in $B$. 
By assumption we also have $d' = m\vee b'\vee m\vee c\vee m\vee b'\vee m$ so that we may assume that $m < b' < d'$. 
Assigning $a'$ to $x$, $b'$ to $y$ and $d$ to $z$, \eqref{dist3} gives $a'\wedge d\wedge a' = a' \wedge (m \vee d \vee b' \vee d \vee m) \wedge a' = a' \wedge (b' \vee d \vee b') \wedge a' = a'\wedge d'\wedge a'$ from which $a\wedge d\wedge a = a\wedge d'\wedge a$ follows by the argument above, and the EMCC is verified. 
Clearly we have a pair of subvarieties. The implications are clear.
\end{proof}

The left-handed and right-handed versions of \eqref{dist3} are respectively: 

\begin{equation}\label{dist3L}\tag{5.3L}
x \wedge (y \vee z \vee (y \wedge x)) = x \wedge (z \vee y)
\end{equation}
\begin{equation}\label{dist3R}\tag{5.3R}
((x \wedge y) \vee z \vee y) \wedge x = (y \vee z) \wedge x.
\end{equation}

We will soon see (cf. Theorem \ref{distlindist} below) that all four special consequences of \eqref{GMD} and \eqref{GJD} - quasi-distributivity, linear distributivity, EMCC and EJCC - are also sufficient for a skew lattice to be distributive. 
It is fortuitous that the two latter conditions are also consequences of symmetry, leading to a major result of this paper, Theorem \ref{distrib}.

\begin{example}\label{nondistrib}\cite{Sp00}
Consider the following left-handed, 9-element example after Spinks where \eqref{GJD} holds but not \eqref{GMD} since: 
$$2\wedge (5\vee 8)\wedge 2 = 2\wedge 4\wedge 2 = 6 \neq 5 = 5 \vee 0 = (2\wedge 5\wedge 2) \vee (2\wedge 8\wedge 2).$$

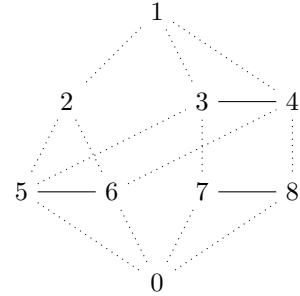
\begin{center}
$\begin{array}{ccc}

\begin{tabular}{ l | c c c c c c c c c}
$\wedge$ & 0 & 1 & 2 & 3 & 4 & 5 & 6 & 7 & 8 \\
\hline
0 & 0 & 0 & 0 & 0 & 0 & 0 & 0 & 0 & 0 \\
1 & 0 & 1 & 2 & 3 & 4 & 5 & 6 & 7 & 8 \\ 
2 & 0 & 2 & 2 & 5 & 6 & 5 & 6 & 0 & 0 \\ 
3 & 0 & 3 & 5 & 3 & 3 & 5 & 5 & 7 & 7 \\
4 & 0 & 4 & 6 & 4 & 4 & 6 & 6 & 8 & 8 \\
5 & 0 & 5 & 5 & 5 & 5 & 5 & 5 & 0 & 0 \\
6 & 0 & 6 & 6 & 6 & 6 & 6 & 6 & 0 & 0 \\
7 & 0 & 7 & 0 & 7 & 7 & 0 & 0 & 7 & 7Ê\\
8 & 0 & 8 & 0 & 8 & 8 & 0 & 0 & 8 & 8 
\end{tabular}

&

\begin{tabular}{ l | c c c c c c c c c}
$\vee$ & 0 & 1 & 2 &  3 &  4&  5&  6&  7&  8 \\
\hline
0 & 0 & 1 & 2 & 3 & 4 & 5 & 6 & 7 & 8 \\
1 & 1 & 1 & 1 & 1 & 1 & 1 & 1 & 1 & 1 \\
2 & 2 & 1 & 2 & 1 & 1 & 2 & 2 & 1 & 1 \\
3 & 3 & 1 & 1 & 3 & 4 & 3 & 4 & 3 & 4 \\
4 & 4 & 1 & 1 & 3 & 4 & 3 & 4 & 3 & 4 \\
5 & 5 & 1 & 2 & 3 & 4 & 5 & 6 & 3 & 4 \\
6 & 6 & 1 & 2 & 3 & 4 & 5 & 6 & 3 & 4 \\
7 & 7 & 1 & 1 & 3 & 4 & 3 & 4 & 7 & 8 \\
8 & 8 & 1 & 1 & 3 & 4 & 3 & 4 & 7 & 8 
\end{tabular}

&

\begin{tikzpicture}[scale=1.2]

  \node (1) at (0,4){$1$} ;
  \node (2) at (-1,3){$2$} ;
  \node (3) at (.5,3){$3$} ;
    \node (4) at (1.5,3){$4$} ;
  \node (5) at (-1.5,2){$5$};
    \node (6) at (-.5,2){$6$};
  \node (7) at (.5,2){$7$} ;
    \node (8) at (1.5,2){$8$} ;
  \node (0) at (0,1){$0$};
  
    \draw[dotted] (1) -- (2) ;
    \draw[dotted] (1) -- (3) ;
    \draw[dotted] (1) -- (4) ;
    \draw[dotted] (3) -- (7) ;
    \draw[dotted] (4) -- (8) ;
    \draw[dotted] (7) -- (0) ;
    \draw[dotted] (8) -- (0) ;
    \draw[dotted] (2) -- (5) ;
    \draw[dotted] (2) -- (6) ;
    \draw[dotted] (5) -- (0) ;
    \draw[dotted] (6) -- (0) ;
    
    \draw[dotted] (5) -- (3) ;
    \draw[dotted] (6) -- (4) ;
    
    \draw (3) -- (4) ;
    \draw (5) -- (6) ;
    \draw (7) -- (8) ;
    
\end{tikzpicture}

\end{array}$
\end{center}
\smallskip
 
\noindent This example is non-upper symmetric. (Indeed, $5\wedge 8 = 0 = 8\wedge 5$ but $5\vee 8 = 4 \neq 3 = 8\vee 5$.) 
Notice that $2 \wedge (8 \vee 5 \vee (8 \wedge 2)) = 2 \wedge (3 \vee 0) = 2 \wedge 3 = 5$, while $2 \wedge (5 \vee 8) = 2 \wedge 4 = 6$, so that \eqref{dist3L} fails.
As mentioned above, Spinks' four examples of order 9 are the first cases where \eqref{dist3}  or its dual do not hold. Mace 4 has shown that all cases of order 10 - 13 where \eqref{dist3} or its dual 
do not hold contain a copy of a SpinksÕ example. The first cases where \eqref{dist3} or its dual do not 
hold, but contain no copy of a Spinks example occur with order 14.
\end{example}

Proceeding on to Theorem \ref{distlindist}, we first further characterize quasi-distributivity in the left-handed case.

\begin{lemma}
A left-handed skew lattice is quasi-distributive if and only if for all $x, y, z\in S$:
\begin{equation}\label{dist4}\tag{5.4}
x\wedge ((y\wedge x)\vee z) = x\wedge (y\vee z)
\end{equation}
\end{lemma}

\begin{proof}
($\Leftarrow$) Clearly, neither $\mathbf M_{3}$ nor $\mathbf N_{5}$ satisfy \eqref{dist4}. 

($\Rightarrow $) If a skew lattice $\mathbf S$ is left-handed, then $y\vee z\geq (y\wedge x)\vee z$ for all $y,z\in S$. 
Indeed,taking joins of both sides both ways gives

$$
\begin{array}{lcl}
y\vee z\vee (y\wedge x)\vee z &=_{\eqref{LH}}& y\vee (y\wedge x)\vee z=_{\eqref{absidentities}} y\vee z \text{   and}\\
(y\wedge x)\vee z\vee y\vee z &=_{\eqref{LH}}& (y\wedge x\wedge y)\vee y\vee z=_{\eqref{absidentities}} y\vee z.
\end{array}
$$

But quasi-distributivity implies both sides of \eqref{dist4} are $\DD$-related and in fact, $\LL$-related.
Thus, we have $(x \wedge (y \vee z)) \wedge (x \wedge ((y \wedge x) \vee z)) = x \wedge (y \vee z)$. But $y \vee z \geq (y \wedge x) \vee z$ gives
$(x \wedge (y \vee z)) \wedge (x \wedge ((y \wedge x) \vee z)) =_{\eqref{LH}} x \wedge (y \vee z) \wedge ((y \wedge x) \vee z) =_{\eqref{absidentities}} x \wedge ((y \wedge x) \vee z)$ so that \eqref{dist4} follows. 
\end{proof}

\begin{theorem}\label{distlindist}
A quasi-distributive, linearly distributive skew lattice satisfies \eqref{GMD} if and only if it satisfies \eqref{dist3}. 
Likewise, a quasi-distributive, linearly distributive skew lattice satisfies \eqref{GJD} if and only if it satisfies the dual of \eqref{dist3}. 
Finally, a skew lattice is distributive if and only if it is quasi-distributive, linearly distributive and satisfies both \eqref{dist3} and its dual.
\end{theorem}

\begin{proof}
Clearly, \eqref{GMD} $\Rightarrow $ \eqref{dist3}. 
To show \eqref{dist3} $\Rightarrow $ \eqref{GMD} under the given conditions, we first consider the left-handed case:

$$
\begin{array}{lcl}
(x \wedge y) \vee (x \wedge z) & =_{\eqref{GMDg}} & x \wedge [(y \wedge x) \vee (z \wedge x)]  \\
                                                 & =_{\eqref{dist4}} & x \wedge [y \vee (z \wedge x)]  \\
                                                 &=_{\eqref{LH}} & x \wedge [(z \wedge x) \vee y \vee (z \wedge x)] \\
                                                 &=_{\eqref{dist4}} & x \wedge [z \vee y \vee (z \wedge x)] \\
                                                 &=_{\eqref{dist3L}} & x \wedge (y \vee z)
\end{array}
$$

\noindent The right-handed case is similar and the general case now follows as usual. 
The second assertion now follows by $\vee -Ê\wedge$ duality and the final assertion from the first two. 
\end{proof}

Recall that a \emph{skew diamond} is a skew lattice with four $\DD$-classes, two being incomparable, say 
$A$ and $B$, and the remaining two being their join and meet $\DD$-classes, say $J$ and $M$. Skew diamonds trivially satisfy the EMCC. Here the nontrivial situations are $A > M < B < J$ where $A\wedge j\wedge A = A$ for all $j\in J$, and $B > M < A < J$ where similar remarks hold. Dually they satisfy the EJCC. Since skew diamonds are clearly quasi-distributive, we have: \emph{a skew diamond is distributive if and only if it is linearly distributive}. Skew diamonds play an important role in the basic theory of skew lattices. See, e.g., their role in \cite{Ka11c} where a number of forbidden algebras are skew diamonds.
\smallskip

Under what reasonable conditions must either \eqref{dist3} or its dual hold? 
They must hold for strictly categorical skew lattices since both sides of \eqref{dist3} are $< x$ but $> x\wedge y\wedge x$. 
We also have:

\begin{proposition}\label{USL}
An upper symmetric skew lattice satisfies \eqref{dist3}.
\end{proposition}

\begin{proof}
We organize the proof for the case when $\mathbf S$ is left-handed in the following steps:

1) Upper symmetry in the left-handed case is characterized by
\begin{equation}\label{LUS}\tag{2.1L}
x\vee y\vee (x\wedge y) = y\vee x.
\end{equation}
Since $x, y \succeq y\wedge x, (y\wedge x) \vee y \vee x$ reduces to $y \vee x$ in the left-handed case. 

2) For all $x, y, z \in S$, $x\vee y \geq (x\vee (y\wedge z)) \wedge (z\vee (y\wedge z))$.

Set $u = (x\vee (y\wedge z)) \wedge  (z\vee (y\wedge z))$. Since $u\wedge y =_{\eqref{id3}} (x\vee (y\wedge z)) \wedge  (y\wedge z) =_{\eqref{absidentities}} y\wedge z \leq y$, we get

$$
\begin{array}{lcl}
x \vee y \vee u &=_{\eqref{LUS}}& x \vee u \vee y \vee (u \wedge y) = x \vee u \vee y \vee (y \wedge z) =_{\eqref{absidentities}} x \vee u \vee y \\
&=_{\eqref{reg}}& x\vee y\vee u\vee y =_{\eqref{id1}} x\vee (y\wedge z)\vee y\vee u\vee y  \\
&=_{(LH)}& x\vee (y\wedge z)\vee u\vee y= x\vee (y\wedge z)\vee ((x\vee (y\wedge z))\wedge (z\vee (y\wedge z)])\vee y \\ &=_{\eqref{absidentities}}& x\vee (y\wedge z)\vee y =_{\eqref{id1}} x\vee y,  \\
\end{array}
$$

which is what needed to be shown in the left-handed case.

3) For all $x,y,z\in S$, $z\wedge [x\vee (y\wedge z)] = z\wedge (x\vee y)\wedge (x\vee (y\wedge z))$.

$$
\begin{array}{lcl}
z \wedge (x \vee (y \wedge z)) &=_{\eqref{absidentities}}&  z \wedge (z \vee (y \wedge z)) \wedge (x \vee (y \wedge z)) \\
&=_{\eqref{LH}}& z \wedge (z \vee (y\wedge z)) \wedge u \\
&=_{(2)}& z \wedge (z \vee (y \wedge z)) \wedge (x \vee y) \wedge u  \\
&=_{\eqref{LH}}& z\wedge (z\vee (y\wedge z))\wedge (x\vee y)\wedge (x\vee (y\wedge z)) \\
&=_{\eqref{absidentities}}&  z \wedge (x \vee y) \wedge (x \vee (y \wedge z)).
\end{array}
$$ 

4) For all $x, y, z \in S$, $z \wedge (x \vee (y \wedge x \wedge z)) = z \wedge (x \vee (y \wedge x))$. Replacing $y$ by $y \wedge x$ in (3) gives

$$
\begin{array}{lcl}
z\wedge (x\vee (y\wedge x\wedge z)) &=& z\wedge (x\vee (y\wedge x))\wedge (x\vee (y\wedge x\wedge z))  \\
&=_{\eqref{id2}}& z \wedge (x \vee (y \wedge x)) \wedge x \wedge (x \vee (y \wedge x \wedge z))  \\
&=_{\eqref{absidentities}} & z\wedge (x\vee (y\wedge x))\wedge x  \\
&=_{\eqref{id2}}& z\wedge (x\vee (y\wedge x)). 
\end{array}
$$

5) Concluding the left-handed case. Replace $x$ with $y \vee x$ in (4). On the left side we get

$$z\wedge (y\vee x\vee (y\wedge (y\vee x)\wedge z)) =_{\eqref{absidentities}} z\wedge (y\vee x\vee (y\wedge z)).$$

On the right side,

$$z \wedge (y \vee x \vee (y \wedge (y \vee x))) =_{\eqref{absidentities}}  z \wedge (y \vee x \vee y) =_{\eqref{LH}} z \wedge (x \vee y).$$

Therefore $z \wedge (x \vee y) = z \wedge (y \vee x \vee (y \wedge z))$ which is \eqref{dist3L} with the variables permuted.

The verification of the right-handed case is similar, and the general case follows. 
\end{proof}

These results and their duals lead to:

\begin{theorem}\label{distrib}

\begin{itemize}
\item[(i)] An upper symmetric skew lattice is $\wedge$-distributive if and only if it is both quasi-distributive and linearly distributive;
\item[(ii)] A lower symmetric skew lattice is $\vee$-distributive if and only if it is both quasi-distributive and linearly distributive.
\item[(iii)] Thus a symmetric skew lattice is distributive if and only if it is both quasi-distributive and linearly distributive.
\end{itemize}

\end{theorem}

Prover9 has also provided proofs of the following results, which we just state.

\begin{theorem} 
A simply cancellative skew lattice is distributive if and only if it is linearly distributive. 
\end{theorem}

\begin{theorem}\label{biconditional} 
A quasi-distributive skew lattice $\mathbf S$ is distributive if it is biconditionally distributive: 
\eqref{GMD} holds for any particular $x, y, z \in S$ iff \eqref{GJD} does.
\end{theorem}

A skew lattice $\mathbf S$ is \emph{relatively distributive} if every quasi-distributive subalgebra of $\mathbf S$ is distributive. 
Such a skew lattice is linearly distributive. More general statements of Theorems \ref{distrib} (iii) and \ref{biconditional} are as follows:

\begin{corollary} 
Biconditionally distributive skew lattices as well as symmetric, linearly distributive skew lattices are relatively distributive.
\end{corollary}

Examples \ref{nondistrib} show that relative distributivity is properly stronger than linear distributivity.
The modular lattice $\mathbf M_{3}$ shows that biconditional distributivity is properly stronger than relative distributivity. 
Indeed any lattice is relatively distributive, but elements $x$, $y$ and $z$ are easily found in $\mathbf M_{3}$ satisfying exactly one of \eqref{GMD} or \eqref{GJD}. 
It can be shown that biconditionally distributive skew lattices form a variety. 
It can also be shown, using Prover9, that a skew lattice is relatively distributive if and only if it is linearly distributive and possesses both the EMCC and EJCC properties. 
Thus relatively distributive skew lattices also form a variety.



\begin{thebibliography}{7}

\bibitem{Ba11}
       A. Bauer and K, Cvetko-Vah, 
       Stone duality for skew Boolean intersection algebras.
       \textbf{Houston Journal of Mathematics} \textbf{39} (2013), 73-109.

\bibitem{BL}
	R. Bignall and J. Leech, 
	Skew boolean algebras and discriminator varieties. 
	\textbf{Algebra Universalis} \textbf{33} (1995), 387Ð398.

\bibitem{BS}
	R. Bignall and M. Spinks,
	Propositional skew Boolean logic. 
	In \textbf{Proc. 26th International Symposium on Multiple-valued Logic}, IEEE Computer Soc. Press (1996), 43-48.
	
\bibitem{Ba40}	
	G. Birkhoff,
	Lattice Theory.
	\textbf{AMS Colloquium Publicatins} (1940).

	
\bibitem{prover}
        W. McCune,
        Mace4/Prover9, Version Dec 2007
        \url{www.cs.unm.edu/~mccune/mace4}

\bibitem{Ka05c} 
	K. Cvetko-Vah,
	Skew lattices of matrices in rings. 
	\textbf{Algebra Universalis} \textbf{53} (2005), 471Ð479.

\bibitem{Ka05}  
	K. Cvetko-Vak,
	Skew lattices in rings. 
	\textbf{PhD thesis}. 
	University of Ljubljana, 2005.

\bibitem{Ka05b} 
	K. Cvetko-Vah,
	Internal decompositions of skew lattices. 
	\textbf{Communications in Algebra} \textbf{35} (2007), 243Ð247.

\bibitem{Ka06} 
	K. Cvetko-Vah,
	A new proof of SpinksÕ Theorem.
	\textbf{Semigroup Forum} \textbf{73} (2006), 267--272.

\bibitem{Ka11c} 
	K. Cvetko-Vah, M. Kinyon, J. Leech, and M. Spinks,
	Cancellation in skew lattices. 
	\textbf{Order} \textbf{28} (2011), 9Ð32.

\bibitem{Ka08} 
	K. Cvetko-Vah and J. Leech,
	Associativity of the $\nabla$ operation on bands in rings.
	\textbf{Semigroup Forum} \textbf{76} (2008), 32--50.

\bibitem{Ka12} 
	K. Cvetko-Vah and J. Leech,
	Rings whose idempotents are multiplicatively closed.
	\textbf{Communications in Algebra} \textbf{40} (2012), 3288--3307.

\bibitem{Ka11} 
	K. Cvetko-Vah and J. Leech,
	On maximal idempotent-closed subrings of $M_{n}(F)$.
	\textbf{International Journal of Algebra and Computation} \textbf{21} (7) (2011), 1097--1110.

\bibitem{CAT} 
	M. Kinyon and J. Leech,
	Categorical skew lattices, 
	\textbf{Order}, in press.

\bibitem{Ku12} 
	G. Kudryavtseva, 
	A refinement of Stone duality to skew Boolean algebras,
	\textbf{Algebra Universalis} \textbf{67} (2012), 397--416.

\bibitem{La02} 
	G. Laslo and J. Leech,
	GreenÕs relations on noncommutative lattices,
	\textbf{Acta Sci. Math. (Szeged)} \textbf{68} (2002), 501--533.

\bibitem{Le89}  
	J. Leech. 
	Skew lattices in rings. 
	\textbf{Algebra Universalis} \textbf{26} (1989), 48Ð72.

\bibitem{Le92}   
	J. Leech. 
	Normal skew lattices. 
	\textbf{Semigroup Forum} \textbf{44} (1992), 1--8.

\bibitem{Le90}  
	J. Leech. 
	Skew boolean algebras. 
	\textbf{Algebra Universalis} \textbf{27} (1990), 497Ð506.

\bibitem{Le93}   
	J. Leech. 
	The geometric structure of skew lattices. 
	\textbf{Trans. Amer. Math. Soc.} \textbf{335} (1993), 823Ð842.

\bibitem{Le96}   
	J. Leech,
	Recent developments in the theory of skew lattices. 
	\textbf{Semigroup Forum} \textbf{52} (1996), 7Ð-24.

\bibitem{Le05}   
	J. Leech,
	Small skew lattices in rings,
	\textbf{Semigroup Forum} \textbf{70} (2005), 307--311.

\bibitem{Le08}   
	J. Leech and M. Spinks, 
	Skew Boolean algebras generated from generalized Boolean algebras,
	\textbf{Algebra Universalis} \textbf{58} (2008), 287--302.

\bibitem{JPC11}
    J. Pita Costa, 
    Coset laws for categorical skew lattices.
    \textbf{Algebra Univers.} \textbf{68} (2012), 75-89.

\bibitem{JPC12} 
    J. Pita Costa, 
    On the coset structure of skew lattices.
    Ph. D. Thesis, University of Ljubljana (2012).

\bibitem{Sp00R} 
    M. Spinks,
    Automated deduction in non-commutative lattice theory, 
    Report 3/98, Monash University, Gippsland School of Computing and Information Technology (1998) 
    
\bibitem{Sp00} 
    M. Spinks,
    On middle distributivity for skew lattices,
    \textbf{Semigroup Forum} \textbf{61} (2000), 341--345. 

\bibitem{Sp06} 
    M. Spinks and R. Veroff,
    Axiomatizing the skew Boolean propositional calculus,
    \textbf{J. Automated Reasoning} \textbf{37} (2006), 3--20.

\end{thebibliography}
\end{document}